\documentclass{article}

\usepackage{PRIMEarxiv}
\usepackage[utf8]{inputenc} 
\usepackage[T1]{fontenc}    
\usepackage{hyperref}       
\usepackage{url}            
\usepackage{booktabs}       
\usepackage{bm}
\usepackage{amssymb}
\usepackage{amsmath}
\usepackage{amsfonts}       
\usepackage{dsfont}
\usepackage{nicefrac}       
\usepackage{microtype}      
\usepackage{lipsum}
\usepackage{fancyhdr}       
\usepackage{graphicx}       
\usepackage{subcaption}
\usepackage{enumitem}
\usepackage{tikz}
\usetikzlibrary{arrows.meta,positioning,calc,decorations.pathreplacing}

\newtheorem{theorem}{Theorem}[section]

\newtheorem{proposition}{Proposition}
\newtheorem{remark}{Remark}

\newenvironment{proof}[1][Proof]{\par\noindent\textbf{#1.} }{\hfill$\square$\par}

\newcounter{assump}

\title{Information Geometry and Asymptotic Theory for SMML Estimators
}

\author{
  Enes Makalic \\
  Faculty of Information Technology \\
  Monash University \\
  Clayton\\
  \texttt{enes.makalic@monash.edu} \\
   \And
  Daniel Francis Schmidt \\
  Faculty of Information Technology \\
  Monash University \\
  Clayton\\
  \texttt{daniel.schmidt@monash.edu} \\
}

\DeclareMathOperator*{\argmax}{arg\,max}
\DeclareMathOperator*{\argmin}{arg\,min}

\begin{document}
\maketitle

\begin{abstract}
Strict minimum message length (SMML) is an information-theoretic coding principle that represents a continuous statistical model by a finite set of assertions and a partition of the sample space. We show that the SMML objective decomposes into assertion entropy and conditional cross-entropy, balancing the cost of identifying an assertion against the cost of encoding data under the assigned model. For any fixed partition, the optimal codepoint for each cell is the model distribution that minimises Kullback--Leibler (KL) divergence from the data distribution restricted to that cell. 
Using the local Fisher--Rao geometry of regular parametric models, we show that, under a high-resolution LAN-scale regime, SMML partitions are asymptotically the pullback, through the maximum-likelihood estimator, of weighted Fisher--Rao Voronoi tessellations in parameter space, with assertion probabilities appearing as additive weights.
%
%
For regular exponential families, SMML codepoints satisfy a moment-matching condition and admit an interpretation as KL/Bregman centroids, while exact SMML cells are pullbacks of convex polyhedra in sufficient-statistic space. Together, these results show that SMML induces a natural information-geometric quantisation linking entropy-based coding, KL projection, and divergence-based Voronoi geometry.
\end{abstract}

\section{Introduction}
\label{sec:introduction}

Minimum message length (MML) inference~\cite{WallaceBoulton68,WallaceBoulton75,WallaceFreeman87,Wallace96,WallaceDowe99c,Wallace05} is an information-theoretic approach to statistical estimation in which data are encoded by a two-part message: an assertion specifying a model or parameter value, followed by an encoding the data conditional on that assertion. Strict minimum message length (SMML)~\cite{WallaceBoulton75,Wallace05} is the ideal finite-codebook version of this principle. It replaces a continuous parameter space by a finite set of assertions, assigns each data set to one assertion, and~chooses both the assertions and the induced partition of the sample space by minimising expected~codelength.

\textls[-15]{In this paper, we study SMML as an entropy--cross-entropy coding principle for finite-resolution statistical inference. The~SMML objective balances assertion entropy against the conditional cross-entropy of encoding data under the assigned model, thereby quantising a continuous statistical model by a finite codebook. 
We show that, in~regular parametric models and under a high-resolution local asymptotic normality (LAN) regime, this quantisation has an asymptotic Fisher--Rao geometry in which the codepoints are characterised by Kullback--Leibler (KL) projections~\cite{KullbackLeibler51} and the induced partitions are asymptotically governed by weighted Fisher--Rao Voronoi tessellations.}
%
%
In exponential families, the~same structure appears explicitly as KL/Bregman centroids satisfying moment-matching~conditions.

Let 
 \(\mathbf{x}=(x_1,\ldots,x_n)\in\mathcal{X}_n\) denote a data set of size \(n\), where \(\mathcal{X}_n\) is assumed to be countable, possibly infinite.
We consider a parametric family 
%
\[
\{p_n(\mathbf{x}\mid\bm{\theta}):\bm{\theta}\in\Theta\subseteq\mathbb{R}^p\},
\]
with prior density \(\pi(\bm{\theta})\). We assume that the data ${\bf x} \in \mathcal{X}_n$ form an independent and identically distributed sample, so that
\(
p_n(\mathbf{x}\mid\bm{\theta})=\prod_{i=1}^n f(x_i\mid\bm{\theta}),
\)
where \(f(\cdot\mid\bm{\theta})\) denotes the per-observation probability mass function. The~marginal distribution of the data is
\begin{align}
\label{eqn:marginal:intro}
r_n(\mathbf{x})
=
\int_{\Theta}
p_n(\mathbf{x}\mid\bm{\theta})\pi(\bm{\theta})\,d\bm{\theta}.
\end{align}
Let
$
\mathcal{P}=\{P_1,\ldots,P_k\}
$
be a finite partition of \(\mathcal{X}_n\) into non-empty cells, and~define the assertion probabilities
\begin{align}
\label{eqn:qj:rx}
q_j
=
\sum_{\mathbf{x}\in P_j}r_n(\mathbf{x}),
\qquad
j=1,\ldots,k.
\end{align}
Throughout, we restrict attention to partitions with $q_j > 0$ for all \(j\).
Under SMML encoding~\cite{WallaceBoulton75,WallaceDowe99a,Wallace05}, the~expected two-part codelength associated with \(\mathcal{P}\) is
\begin{align}
\label{eqn:smml}
\mathcal{I}(\mathcal{P})
&=
\underbrace{-\sum_{j=1}^k q_j\log q_j}_{H(Q):\ \text{assertion entropy}}
\underbrace{-\sum_{j=1}^k\sum_{\mathbf{x}\in P_j}
r_n(\mathbf{x})\log p_n(\mathbf{x}\mid\bm{\theta}_j^*)}_{\text{expected detail codelength}}  \\
&= H(r_n) + \sum_{j=1}^{k} q_j D_{\mathrm{KL}}
\left(
\bar r_{j,n}
\,\middle\|\,
p_n(\cdot\mid\bm{\theta}_j^*)
\right) \label{eqn:smml:twopart},
\end{align}
where \(Q=(q_1,\ldots,q_k)\) is the distribution of the assertion message, $H(\cdot)$ is the entropy, $D_{\mathrm{KL}}$ denotes Kullback--Leibler (KL) divergence and
\[
\bar r_{j,n}(\mathbf{x})
=
\frac{r_n(\mathbf{x})}{q_j}\mathds{1}\{\mathbf{x}\in P_j\}
\]
is the conditional marginal distribution in cell \(P_j\). Unless~otherwise stated, logarithms are natural, so codelengths are measured in nats. The~first term in \eqref{eqn:smml} is the Shannon entropy of the assertion distribution, while the second is the expected detail codelength incurred when data in cell \(P_j\) are encoded using \(p_n(\cdot\mid\bm{\theta}_j^*)\).

For each cell \(P_j\), the~associated SMML codepoint is chosen to minimise the expected detail codelength within that cell:
\begin{align}
\label{eqn:smml:estimate}
\bm{\theta}_j^*
&=
\argmin_{\bm{\theta}\in\Theta}
\sum_{\mathbf{x}\in P_j}
r_n(\mathbf{x})\{-\log p_n(\mathbf{x}\mid\bm{\theta})\}  
=
\argmax_{\bm{\theta}\in\Theta}
\sum_{\mathbf{x}\in P_j}
r_n(\mathbf{x})\log p_n(\mathbf{x}\mid\bm{\theta}).
\end{align}
%
%
From \eqref{eqn:smml:twopart}, the~optimal SMML codepoints satisfy
\begin{align}
\label{eqn:klprojection:intro}
\bm{\theta}_j^*
=
\argmin_{\bm{\theta}\in\Theta}
D_{\mathrm{KL}}
\left(
\bar r_{j,n}
\,\middle\|\,
p_n(\cdot\mid\bm{\theta})
\right), \quad j = 1,\ldots, k.
\end{align}
Thus, for~a fixed partition, each SMML codepoint is the KL, or~information, projection of the normalised cellwise distribution onto the parametric model family in the sense of Csisz\'ar~\cite{Csiszar75}. The~remaining optimisation problem is to choose the partition itself,
\begin{align}
\label{eqn:smml:partition}
\mathcal{P}^*
=
\argmin_{\mathcal{P}\in\Pi}
\mathcal{I}(\mathcal{P}),
\end{align}
where \(\Pi\) is a prescribed class of admissible partitions of
\(\mathcal{X}_n\). The~resulting partition and codepoints \(\Theta^*=\{\bm{\theta}_1^*,\ldots,\bm{\theta}_k^*\}\) determine the SMML estimator~\cite{WallaceBoulton75,WallaceDowe99a,Wallace05}. For~the optimal partition, we write \(q_j^*\), or~\(q_{j,n}^*\) in sample-size-dependent settings for~the corresponding assertion~probabilities.

We make four main contributions. First, we characterise SMML codepoints as KL projections of conditional cell distributions onto the model family, making explicit the entropy--cross-entropy structure of the SMML objective. Second, under~a high-resolution local asymptotic normality (LAN) regime, we show that SMML partitions are asymptotically the pullback, through the maximum-likelihood estimator (MLE), of~weighted Fisher--Rao Voronoi tessellations in parameter space. Third, in~the same regime, we derive an asymptotic weighted-average representation of the SMML codepoints and use it to establish consistency and the usual \(n^{-1/2}\) parametric rate. Fourth, for~exponential families, we show that SMML codepoints satisfy moment matching and that the induced partition is polyhedral in sufficient-statistic space, connecting SMML to KL/Bregman centroids and dually flat geometry.
%

Our exponential-family results are closely related to Dowty's work on SMML estimators with continuous sufficient statistics~\cite{Dowty13}. Dowty showed that the regions of such estimators are convex polytopes described in terms of assertions and coding probabilities. We obtain the corresponding polyhedral structure in the countable sample-space setting and connect it to KL projection, Bregman centroids, and~the asymptotic Fisher--Rao quantisation developed here. The~covering interpretation used below is also related to the distinguishable-distribution perspective of Balasubramanian and of Myung~et~al.~\cite{Balasubramanian97,MyungEtAl00}, where Fisher--Rao volume is interpreted as a count of statistically distinguishable probability distributions. Whereas that work concerns model selection and geometric model complexity, the~present paper studies how SMML induces a finite entropy-coded codebook and partition within a fixed parametric~model.

The remainder of the paper is organised as follows.
Section~\ref{sec:regularity} introduces the regularity conditions and the asymptotic quantisation framework. Section~\ref{sec:fisherrao} establishes the connection between SMML partitions and weighted Fisher--Rao Voronoi tessellations, and~derives the weighted-average representation of SMML codepoints. Section~\ref{sec:consistency} gives the statistical consequences, including the consistency and the \(n^{-1/2}\) convergence rate.
Section~\ref{sec:expfam} specialises the theory to exponential families, where codepoints become KL/Bregman centroids and the partition is polyhedral in sufficient-statistic space. Section~\ref{sec:discussion} discusses connections with entropy-based coding, information-theoretic quantisation, and~possible extensions.
\section{Asymptotic Quantisation Framework and Regularity~Conditions}
\label{sec:regularity}
The SMML objective in~\eqref{eqn:smml} defines a finite-resolution
information-theoretic approximation to a continuous statistical model. Each assertion corresponds to a codepoint in parameter space, and~each data set is assigned to one of these codepoints through the induced partition of the data space \(\mathcal{X}_n\). Thus, an SMML code simultaneously induces a partition of the sample space and a finite quantisation of the parameter space \(\Theta\).

For a sample of size \(n\), let
\[
\Theta_n^*
=
\{\bm{\theta}_{1,n}^*,\ldots,\bm{\theta}_{k_n,n}^*\}
\]
denote the SMML codebook, where \(k_n\) is the number of assertions. The~asymptotic regime considered here is a high-resolution regime in which \(k_n\) may grow with \(n\) and the local mesh of the induced quantisation shrinks as the data become more informative. This is the regime in which a finite SMML codebook can approximate the underlying continuous~model.

The relevant local geometry is governed by Kullback--Leibler
divergence~\cite{KullbackLeibler51}. For~nearby parameter values,
\[
D_{\mathrm{KL}}
\!\left(
f(\cdot\mid\bm{\theta})
\,\middle\|\,
f(\cdot\mid\bm{\theta}+d\bm{\theta})
\right)
=
\frac12
d\bm{\theta}^{\prime}
{\bf J}_1(\bm{\theta})
d\bm{\theta}
+
o(\|d\bm{\theta}\|^2),
\]
where \({\bf J}_1(\bm{\theta})\) is the per-observation Fisher information matrix. Hence, the natural local metric for analysing SMML codepoints is the per-observation Fisher--Rao metric. Since the full-sample Fisher information is \(n{\bf J}_1(\bm{\theta})\), displacements of order \(n^{-1/2}\) in this metric correspond to constant-order changes in full-sample likelihood~geometry.

This \(n^{-1/2}\) scale is the statistical scale of distinguishability. Balasubramanian~\cite{Balasubramanian97} and Myung~et~al.~\cite{MyungEtAl00} show that Fisher--Rao volume may be interpreted as counting statistically distinguishable probability distributions. For~a finite sample size, nearby distributions inside a Fisher information ellipsoid are difficult to distinguish, and~these indistinguishability
neighbourhoods shrink at the parametric rate as \(n\) grows. The~same scale appears in local asymptotic normality (LAN). Under~the local reparameterisation
\(
\bm{\theta}=\bm{\theta}_0+\mathbf{h}/\sqrt n
\),
the log-likelihood ratio admits a non-degenerate quadratic Gaussian limit, with~curvature given by the Fisher~information.

In the SMML setting, the~assertion codebook may therefore be interpreted as a finite entropy-coded covering of the model at a chosen Fisher--Rao resolution. On~a compact set \(K\subset\Theta\), a~Fisher--Rao ball of radius \(\delta_n\) has volume of order \(\delta_n^p\). Thus, up~to curvature and boundary effects, the~number of assertions required to cover \(K\) at uniform resolution \(\delta_n\) is heuristically
\[
N_K(\delta_n)
\asymp
\frac{\operatorname{Vol}_F(K)}{\delta_n^p},
\qquad
\operatorname{Vol}_F(K)
=
\int_K |{\bf J}_1(\bm{\theta})|^{1/2}\,d\bm{\theta}.
\]
More generally, for~a non-uniform mesh with local radius
\(\delta_n(\bm{\theta})\),
\[
k_n(K)
\asymp
\int_K
\frac{|{\bf J}_1(\bm{\theta})|^{1/2}}
{\delta_n(\bm{\theta})^p}
\,d\bm{\theta}.
\]
The natural benchmark resolution is the LAN-scale \(\delta_n(\bm{\theta})\asymp n^{-1/2}\), which yields  \(k_n(K)\asymp n^{p/2}\operatorname{Vol}_F(K)\). Thus, \(n^{p/2}\) is the benchmark order for a LAN-scale codebook on compact subsets.

The results below are established under the following regularity conditions. Assumptions \ref{ass:A1}--\ref{ass:A6} are standard likelihood regularity conditions; see, e.g.,~\cite{Vaart98}. Assumption \ref{ass:A7} specifies the local high-resolution structure of the SMML quantisation by placing the selected SMML codepoint within an \(O(n^{-1/2})\) Fisher--Rao neighbourhood of the MLE and requiring local quadratic Fisher--Rao loss inside the corresponding cells.
\begin{enumerate}[label=(A\arabic*)]

\item \label{ass:A1}
The parameter space $\Theta$ is an open subset of $\mathbb{R}^p$.

\item \label{ass:A2}
The true parameter $\bm{\theta}_0$ belongs to $\operatorname{int}(\Theta)$.

\item \label{ass:A3}
The support of $p_n(\mathbf{x}|\bm{\theta})$ does not depend on $\bm{\theta}\in\Theta$. The~prior $\pi(\bm{\theta})$ is a proper density on $\Theta$, continuous and strictly positive on a neighbourhood of the true parameter $\bm{\theta}_0$. For~each $n$ and each $\mathbf{x}\in\mathcal{X}_n$, the~marginal distribution
\[
r_n(\mathbf{x})
=
\int_\Theta
p_n(\mathbf{x}\mid\bm{\theta})
\pi(\bm{\theta})
\,d\bm{\theta}
\]
is~finite.

\item \label{ass:A4}
For each $\mathbf{x}\in\mathcal{X}_n$, the~map $\bm{\theta}\mapsto \log p_n(\mathbf{x}|\bm{\theta})$ is three times continuously differentiable on $\Theta$. Moreover, for~all data sets $\mathbf{x}$ in the SMML cells under consideration, the~maximum likelihood estimator $\hat{\bm{\theta}}(\mathbf{x})$ exists and belongs to $\Theta$.
The model is identifiable and satisfies standard regularity conditions under which the maximum likelihood estimator is consistent at the usual parametric~rate.

\item \label{ass:A5}
Let
\[
\ell_n(\mathbf{x}|\bm{\theta})
=
\log p_n(\mathbf{x}|\bm{\theta})
\]
denote the log-likelihood for an i.i.d. data set of size $n$, and~let
\begin{align}
{\bf J}_1(\bm{\theta})
=
\mathbb{E}_{\bm{\theta}}
\!\left[
\nabla_{\bm{\theta}}\log f(X_1|\bm{\theta})
\nabla_{\bm{\theta}}\log f(X_1|\bm{\theta})^\prime
\right]
\end{align}
denote the per-observation Fisher information matrix. Thus, the full-sample Fisher information is
\[
{\bf J}_n(\bm{\theta})
=
n\,{\bf J}_1(\bm{\theta}).
\]
Assume that ${\bf J}_1(\bm{\theta})$ exists and is locally uniformly positive definite on $\Theta$: for every compact set $K\subset\Theta$, there exist constants $0<c_1(K)\le c_2(K)<\infty$ such that
\begin{align}
\label{eqn:fisherbound}
c_1(K)\,{\bf I}_p
\preceq
{\bf J}_1(\bm{\theta})
\preceq
c_2(K)\,{\bf I}_p,
\qquad
\bm{\theta}\in K,
\end{align}
where ${\bf I}_p$ is the $p\times p$ identity matrix and $\preceq$ denotes the Loewner~order.

\item \label{ass:A6}
Differentiation under the integral sign is valid whenever it is used~below.

\item \label{ass:A7}
%
Fix a compact set \(K\subset\Theta\) on which~\ref{ass:A5} holds. The~conditional distributions of \(\sqrt n\,d_F(\hat{\bm{\theta}}(\mathbf{X}_n),\bm{\theta}_{j,n}^*)\) given \(\mathbf{X}_n\in P_{j,n}^*\) are uniformly tight over SMML cells with \(\bm{\theta}_{j,n}^*\in K\): for every \(\epsilon>0\), there exists \(M<\infty\) such that, for~all sufficiently large \(n\),
\[
\sup_{j:\,\bm{\theta}_{j,n}^*\in K}
\mathbb{P}\!\left(
\sqrt n\,
d_F\!\left(\hat{\bm{\theta}}(\mathbf{X}_n),\bm{\theta}_{j,n}^*\right)>M
\,\middle|\,
\mathbf{X}_n\in P_{j,n}^*
\right)
<\epsilon.
\]
Moreover,
\[
\ell_n(\mathbf{X}_n\mid\bm{\theta}_{j,n}^*)
=
\ell_n(\mathbf{X}_n\mid\hat{\bm{\theta}}(\mathbf{X}_n))
-
\frac n2 d_F^2\!\left(\hat{\bm{\theta}}(\mathbf{X}_n),\bm{\theta}_{j,n}^*\right)
+
o_p(1),
\]
conditional on \(\mathbf{X}_n\in P_{j,n}^*\), uniformly over such cells.
\end{enumerate}
Assumption~\ref{ass:A7} should be understood as a conditional high-resolution regularity condition. It places the SMML codebook in an LAN-scale regime in which the selected codepoint is \(O(n^{-1/2})\) from the MLE in Fisher--Rao distance, and~likelihood differences inside the corresponding cells are governed by local quadratic Fisher--Rao loss. 
Consequently, the~asymptotic results in Sections~\ref{sec:fisherrao} and
\ref{sec:consistency} are conditional on this high-resolution regime.
We do not claim here that this regime follows automatically from global SMML optimality for every regular model.
%
\begin{remark}[Rate--distortion interpretation of the high-resolution regime]
The SMML objective decomposes into an assertion-entropy term and a conditional Kullback--Leibler distortion term, up~to partition-independent constants. The~assertion entropy plays the role of rate, while the conditional KL term plays the role of distortion. In~local coordinates, the~KL distortion is quadratically approximated by Fisher information, so the high-resolution SMML problem is naturally compared with high-rate quantisation under quadratic Fisher--Rao distortion. From~this viewpoint, \ref{ass:A7} places the SMML estimator in the regime where the entropy--cross-entropy trade-off is resolved at the LAN~scale.

A local split-optimality heuristic suggests why this scale is natural. Splitting a cell of assertion mass \(q_j\) incurs an entropy cost of order \(q_j\), whereas, under~local quadratic Fisher--Rao loss, the~potential reduction in detail codelength is of order \(q_j n r_j^2\) for a cell of effective Fisher--Rao radius \(r_j\). Balancing these terms heuristically gives \(r_j^2\asymp n^{-1}\). A~rigorous derivation of this LAN-scale localisation directly from global SMML optimality, including boundary and heterogeneity effects, is left for future work.
\end{remark}
\section{Fisher--Rao Geometry of~SMML}
\label{sec:fisherrao}
We now derive the local geometric form of the SMML decision rule under the high-resolution regime described in \ref{ass:A7}. In~this regime, each data set \({\bf x}\) is first mapped to parameter space by its maximum-likelihood estimate (MLE) \(\hat{\bm{\theta}}({\bf x})\). The~selected SMML assertion is then the nearby codepoint that gives the shortest asymptotic two-part message. Since the loss from replacing \(\hat{\bm{\theta}}({\bf x})\) by a nearby codepoint is governed locally by Fisher--Rao quadratic loss, the~assignment rule takes the form of a weighted Fisher--Rao Voronoi decision rule, with~the assertion probabilities appearing as additive weights.

Throughout this section, \(d_F(\bm{\theta}_1,\bm{\theta}_2)\) denotes the per-observation Fisher--Rao distance induced by the metric \({\bf J}_1(\bm{\theta})\). For~nearby parameter values,
\begin{align}
\label{eqn:FRlocal}
d_F^2(\bm{\theta}_1,\bm{\theta}_2)
=
(\bm{\theta}_1-\bm{\theta}_2)^\prime
{\bf J}_1(\bar{\bm{\theta}})
(\bm{\theta}_1-\bm{\theta}_2)
+
o\!\left(\|\bm{\theta}_1-\bm{\theta}_2\|^2\right),
\end{align}
where \(\bar{\bm{\theta}}\) lies on the line segment joining \(\bm{\theta}_1\) and \(\bm{\theta}_2\). The~use of the per-observation metric separates the intrinsic geometry of the model from the sample-size scaling of the full likelihood, whose curvature is \(n{\bf J}_1(\bm{\theta})\).

\begin{theorem}[SMML partitions and weighted Fisher--Rao Voronoi cells]
\label{thm:smml:fr:equivalence}
Let
\[
\mathcal{P}_n^{*}
=
\{P_{j,n}^*\}_{j=1}^{k_n},
\qquad
\bm{\Theta}_n^*
=
\{\bm{\theta}_{j,n}^*\}_{j=1}^{k_n}
\]
denote an optimal SMML partition and its associated codepoints, with~assertion probabilities
\[
q_{j,n}^*
=
\sum_{\mathbf{x}\in P_{j,n}^*} r_n(\mathbf{x}) > 0.
\]
Define the weighted Fisher--Rao Voronoi cells
\begin{align}
\label{eqn:weightedvoronoi}
V_{j,n}
=
\left\{
\bm{\theta}\in\Theta :
d_F^2(\bm{\theta},\bm{\theta}_{j,n}^*)+\omega_{j,n}
\le
d_F^2(\bm{\theta},\bm{\theta}_{\ell,n}^*)+\omega_{\ell,n},
\ \forall \ell
\right\},
\end{align}
where
\[
\omega_{j,n}
=
-\frac{2}{n}\log q_{j,n}^*.
\]
Under~\ref{ass:A1}--\ref{ass:A7}, and~assuming that the local quadratic approximation in \ref{ass:A7} holds uniformly for all assertions that are locally competitive with the selected assertion, the~optimal SMML partition is, up~to tie-breaking on a set of probability \(o(1)\), the~pullback of this weighted Fisher--Rao Voronoi tessellation under the MLE map
\begin{align}
\label{eqn:smml:Fisher:Voronoi}
\mathbb{P}\!\left(
\mathbf{X}_n\in P_{j,n}^*
\iff
\hat{\bm{\theta}}(\mathbf{X}_n)\in V_{j,n}
\right)
\to
1,
\end{align}
for each $j = 1,\ldots,k_n$. Here, an~assertion is called locally competitive for \(x\) if its two-part codelength differs from the minimum two-part codelength by \(O(1)\).
Moreover, the~pairwise boundary between two locally competing cells, \(j\) and
\(\ell\), satisfies
\begin{align}
\label{eqn:smml:cell:boundary}
\frac{n}{2}
\left\{
d_F^2(\bm{\theta},\bm{\theta}_{j,n}^*)
-
d_F^2(\bm{\theta},\bm{\theta}_{\ell,n}^*)
\right\}
=
\log\frac{q_{j,n}^*}{q_{\ell,n}^*}
+
o(1),
\end{align}
uniformly on compact subsets where the local quadratic approximation is valid. The~corresponding data-space boundary is obtained by pulling this boundary back under \(\mathbf{x}\mapsto\hat{\bm{\theta}}(\mathbf{x})\).
\end{theorem}

\begin{proof}
Fix the optimal codebook \(\{(\bm{\theta}_{j,n}^*,q_{j,n}^*)\}_{j=1}^{k_n}\). Conditional on this codebook, assigning \(\mathbf{x}\) to assertion \(j\) gives the codelength
\[
\Lambda_{j,n}(\mathbf{x})
=
-\log q_{j,n}^*
-
\log p_n(\mathbf{x}\mid\bm{\theta}_{j,n}^*).
\]
Thus, the induced partition assigns \(\mathbf{x}\) to an index minimising \(\Lambda_{j,n}(\mathbf{x})\), up~to ties.
Let \(\hat{\bm{\theta}}=\hat{\bm{\theta}}({\bf x})\). 
By~\ref{ass:A7},
\[
d_F\!\left(\hat{\bm{\theta}}({\bf X}_n),\bm{\theta}_{j,n}^*\right)
=
O_p(n^{-1/2})
\qquad
\text{conditional on } {\bf X}_n\in P_{j,n}^*.
\]
Moreover,
\[
-\log p_n({\bf X}_n\mid\bm{\theta}_{j,n}^*)
=
-\log p_n({\bf X}_n\mid\hat{\bm{\theta}}({\bf X}_n))
+
\frac n2 d_F^2\!\left(\hat{\bm{\theta}}({\bf X}_n),\bm{\theta}_{j,n}^*\right)
+
o_p(1),
\]
conditional on \({\bf X}_n\in P_{j,n}^*\), uniformly over such cells. By~the additional uniform local-competition assumption in the statement of
the theorem, the~same expansion holds for all assertions whose two-part
codelengths are within \(O(1)\) of the minimum, and~hence for all assertions that can affect the asymptotic decision boundary.
Therefore,
\[
\Lambda_{j,n}({\bf X}_n)
=
-\log q^*_{j,n}
-
\log p_n({\bf X}_n \mid \hat{\bm{\theta}})
+
\frac n2 d_F^2(\hat{\bm{\theta}},\bm{\theta}^*_{j,n})
+
o_p(1).
\]
Since the middle term does not depend on \(j\), the~selected assertion asymptotically minimises
\[
-\log q^*_{j,n}
+
\frac n2 d_F^2(\hat{\bm{\theta}},\bm{\theta}^*_{j,n}),
\]
or equivalently
\[
d_F^2(\hat{\bm{\theta}},\bm{\theta}^*_{j,n})-\frac{2}{n}\log q^*_{j,n},
\]
which is the weighted Fisher--Rao Voronoi rule. Equating the asymptotic codelengths for two locally competing assertions \(j\) and \(\ell\) yields \eqref{eqn:smml:cell:boundary}. Pulling this boundary back through the MLE map gives the corresponding boundary in data space.
\end{proof}

\begin{remark}[Uniform assertion probabilities]
If assertion probabilities are uniformly asymptotically equal, in~the sense that
\[
\sup_{1\le j,\ell\le k_n} \left| \log \frac{q_{j,n}^*}{q_{\ell,n}^*}\right| \to 0,
\]
then
\[
\omega_{j,n}-\omega_{\ell,n}
=
-\frac{2}{n}\log\frac{q_{j,n}^*}{q_{\ell,n}^*}
=
o(n^{-1}).
\]
Hence, the weighted Fisher--Rao Voronoi tessellation reduces asymptotically to the ordinary Fisher--Rao Voronoi tessellation. In~this case, local SMML assignment is governed primarily by Fisher--Rao distance.
\end{remark}

Theorem~\ref{thm:smml:fr:equivalence} shows that SMML partitions are governed locally by Fisher--Rao geometry. Figure~\ref{fig:smml:pullback} depicts the SMML assignment mechanism, showing how the partition of data space arises due to the pullback, 
 through the MLE map, of~weighted Fisher--Rao Voronoi cells in parameter space. The~next result identifies the corresponding location of the codepoints in the same high-resolution regime. It shows that, within~each SMML cell, the~associated codepoint is asymptotically the marginal-predictive weighted average of the MLEs assigned to that cell.

\begin{figure}
\begin{center}
\begin{tikzpicture}[
    >=Latex,
    font=\small,
    panel/.style={draw=black!70, line width=0.7pt, rounded corners=2pt},
    boundary/.style={draw=black!65, dashed, line width=0.9pt},
    maparrow/.style={->, line width=0.9pt, black!75},
    assignarrow/.style={->, line width=0.8pt, black!55, dashed},
    datapoint/.style={circle, fill=black, inner sep=1.35pt},
    mlepoint/.style={circle, fill=black!70, inner sep=1.25pt},
    codepoint/.style={circle, fill=black, draw=white, line width=0.4pt, inner sep=2.1pt},
    lab/.style={font=\small},
    smalllab/.style={font=\scriptsize}
]

\def\leftx{0}
\def\rightx{7.4}
\def\panelw{4.6}
\def\panelh{3.3}

\begin{scope}
    \clip (\leftx,-1.65) rectangle ++(\panelw,\panelh);

    \fill[blue!13]   (\leftx,-1.65) rectangle ++(1.75,\panelh);
    \fill[green!13]  (\leftx+1.55,-1.65) -- ++(3.05,0) -- ++(0,3.3) -- ++(-1.45,0) -- cycle;
    \fill[orange!18] (\leftx+1.25,-1.65) -- ++(3.35,0) -- ++(0,1.55) -- ++(-3.0,1.75) -- cycle;

    \draw[boundary] (\leftx+1.55,-1.65) -- (\leftx+1.75,1.65);
    \draw[boundary] (\leftx+1.25,-1.65) -- (\leftx+4.25,1.65);
\end{scope}

\draw[panel] (\leftx,-1.65) rectangle ++(\panelw,\panelh);
\node[lab] at (\leftx+2.3,1.95) {Data space \(\mathcal X_n\)};
\node[smalllab, blue!55!black]   at (\leftx+0.75,1.25) {\(P_{1,n}^*\)};
\node[smalllab, green!45!black]  at (\leftx+3.55,1.35) {\(P_{2,n}^*\)};
\node[smalllab, orange!70!black] at (\leftx+2.75,-1.35) {\(P_{3,n}^*\)};

\node[datapoint] (x1) at (\leftx+0.70,0.55) {};
\node[datapoint] (x2) at (\leftx+1.05,-0.45) {};
\node[datapoint] (x3) at (\leftx+3.50,0.65) {};
\node[datapoint] (x4) at (\leftx+3.05,-0.70) {};
\node[datapoint] (x5) at (\leftx+2.20,-0.95) {};

\node[smalllab] at ($(x1)+(-0.15,0.15)$) {\(\mathbf{x}\)};
\node[smalllab, align=center] at (\leftx+2.30,-2.05)
{SMML cells in data space};

\begin{scope}
    \clip (\rightx,-1.65) rectangle ++(\panelw,\panelh);

    \fill[blue!13]   (\rightx,-1.65) rectangle ++(1.72,\panelh);
    \fill[green!13]  (\rightx+1.55,-1.65) -- ++(3.05,0) -- ++(0,3.3) -- ++(-1.55,0) -- cycle;
    \fill[orange!18] (\rightx+1.20,-1.65) -- ++(3.40,0) -- ++(0,1.45) -- ++(-2.85,1.85) -- cycle;

    \draw[boundary] (\rightx+1.55,-1.65) .. controls (\rightx+1.80,-0.2) .. (\rightx+1.70,1.65);
    \draw[boundary] (\rightx+1.20,-1.65) .. controls (\rightx+2.35,-0.25) .. (\rightx+4.05,1.65);
\end{scope}

\draw[panel] (\rightx,-1.65) rectangle ++(\panelw,\panelh);
\node[lab] at (\rightx+2.3,1.95) {Parameter space \(\Theta\)};
\node[smalllab, blue!55!black]   at (\rightx+0.72,1.25) {\(V_{1,n}\)};
\node[smalllab, green!45!black]  at (\rightx+4.15,1.25) {\(V_{2,n}\)};
\node[smalllab, orange!70!black] at (\rightx+3.85,-1.25) {\(V_{3,n}\)};

\node[codepoint] (t1) at (\rightx+0.78,0.15) {};
\node[codepoint] (t2) at (\rightx+3.85,0.60) {};
\node[codepoint] (t3) at (\rightx+2.65,-0.85) {};

\node[smalllab, anchor=south west] at ($(t1)+(-0.70,-0.35)$) {\(\bm{\theta}_{1,n}^*\)};
\node[smalllab, anchor=south west] at ($(t2)+(0.01,-0.10)$) {\(\bm{\theta}_{2,n}^*\)};
\node[smalllab, anchor=north west] at ($(t3)+(0.02,-0.02)$) {\(\bm{\theta}_{3,n}^*\)};

\node[mlepoint] (m1) at (\rightx+0.95,0.72) {};
\node[mlepoint] (m2) at (\rightx+1.10,-0.35) {};
\node[mlepoint] (m3) at (\rightx+3.25,0.20) {};
\node[mlepoint] (m4) at (\rightx+2.95,-0.35) {};
\node[mlepoint] (m5) at (\rightx+2.35,-1.05) {};

\node[smalllab] at ($(m1)+(0.25,0.25)$) {\(\hat{\bm{\theta}}(\mathbf{x})\)};

\node[smalllab, align=center] at (\rightx+2.30,-2.05)
{Weighted Fisher--Rao Voronoi cells};

\draw[maparrow] (x1) to[bend left=10] node[above, sloped, smalllab] {} (m1);
\draw[maparrow] (x2) to[bend right=6] (m2);
\draw[maparrow] (x3) to[bend left=5] (m3);
\draw[maparrow] (x4) to[bend right=5] (m4);
\draw[maparrow] (x5) to[bend right=10] (m5);

\node[lab, align=center] at (5.95,1.40)
{\(\mathbf{x}\mapsto\hat{\bm{\theta}}(\mathbf{x})\)\\[-1mm]
\scriptsize MLE map};

\draw[assignarrow] (m1) -- (t1);
\draw[assignarrow] (m2) -- (t1);
\draw[assignarrow] (m3) -- (t2);
\draw[assignarrow] (m4) -- (t3);
\draw[assignarrow] (m5) -- (t3);

\node[smalllab, align=center] at (\rightx+2.30,2.55)
{\(\displaystyle
\arg\min_j
\left\{
d_F^2\!\left(\hat{\bm{\theta}},\bm{\theta}_{j,n}^*\right)
-\frac{2}{n}\log q_{j,n}^*
\right\}\)};

\node[datapoint] (legx) at (\leftx+0.30,-2.55) {};
\node[smalllab, anchor=west] at ($(legx)+(0.15,0)$) {data set};

\node[mlepoint] (legm) at (\leftx+1.65,-2.55) {};
\node[smalllab, anchor=west] at ($(legm)+(0.15,0)$) {MLE image};

\node[codepoint] (legt) at (\leftx+3.45,-2.55) {};
\node[smalllab, anchor=west] at ($(legt)+(0.15,0)$) {SMML codepoint};

\draw[boundary] (\rightx+0.10,-2.55) -- ++(0.45,0);
\node[smalllab, anchor=west] at (\rightx+0.65,-2.55) {weighted Voronoi boundary};

\end{tikzpicture}

\caption{
SMML 
 as the pullback of a weighted Fisher--Rao Voronoi tessellation. The~(\textbf{left}) panel shows the induced SMML partition of data space \(\mathcal X_n\), with~cells \(P_{j,n}^*\). Data sets are mapped to parameter space by the MLE map \(\mathbf{x}\mapsto\hat{\bm{\theta}}(\mathbf{x})\). The~(\textbf{right}) panel shows the corresponding weighted Fisher--Rao Voronoi cells \(V_{j,n}\) around the active SMML codepoints \(\bm{\theta}_{j,n}^*\). Dashed curves indicate pairwise weighted Voronoi boundaries. The~SMML cells in data space are the pullbacks of the parameter-space cells under the MLE map.
}
\label{fig:smml:pullback}
\end{center}
\end{figure}

\begin{theorem}[SMML codepoints as weighted averages of MLEs]
\label{thm:smml:weighted:mle}
With the notation of Theorem~\ref{thm:smml:fr:equivalence}, define
\[
w_{j,n}(\mathbf{x})
=
\frac{r_n(\mathbf{x})}{q_{j,n}^*},
\qquad
\mathbf{x}\in P_{j,n}^*.
\]
Under~\ref{ass:A1}--\ref{ass:A7}, fix a compact set
\(K\subset\Theta\) on which~\ref{ass:A5} holds, and~consider any sequence of cells such that \(\bm{\theta}_{j,n}^*\in K\). Suppose, in~addition, that the following score linearisation holds uniformly for \(\mathbf{x} \in P^*_{j,n}\):
\[
\nabla_{\bm{\theta}} \ell_n(\mathbf{x}\mid\bm{\theta}_{j,n}^*)
=
-n {\bf J}_1(\bm{\theta}_{j,n}^*)
\{\bm{\theta}_{j,n}^*-\hat{\bm{\theta}}(\mathbf{x})\}
+
\mathbf{r}_{j,n}(\mathbf{x}),
\]
with
\[
\sup_{\mathbf{x}\in P_{j,n}^*}\|\mathbf{r}_{j,n}(\mathbf{x})\|=o(\sqrt n).
\]
%
%
Then
\[
\bm{\theta}_{j,n}^*
=
\sum_{\mathbf{x}\in P_{j,n}^*}
w_{j,n}(\mathbf{x})\hat{\bm{\theta}}(\mathbf{x})
+
\bm{\varepsilon}_{j,n},
\qquad
\|\bm{\varepsilon}_{j,n}\|=o(n^{-1/2}).
\]
Equivalently,
\[
\bm{\theta}_{j,n}^*
=
\mathbb{E}_{r_n}
\left[
\hat{\bm{\theta}}(\mathbf{X}_n)
\,\middle|\,
\mathbf{X}_n\in P_{j,n}^*
\right]
+
o(n^{-1/2}).
\]
\end{theorem}

\begin{proof}
For a fixed cell $j$, define the cellwise expected log-likelihood
\begin{align}
\label{eqn:thm2:Ljn}
L_{j,n}(\bm{\theta})
=
\sum_{\mathbf{x}\in P_{j,n}^*}
r_n(\mathbf{x})\,
\ell_n(\mathbf{x}|\bm{\theta}).
\end{align}
By definition, $\bm{\theta}_{j,n}^*$ maximises $L_{j,n}(\bm{\theta})$. Since $\Theta$ is open \ref{ass:A1} and $\bm{\theta}_{j,n}^*\in K\subset\Theta$ by hypothesis, the~first-order condition holds:
%
\begin{align}
\label{eqn:thm2:first:order}
0
=
\left.
\nabla_{\bm{\theta}}L_{j,n}(\bm{\theta})
\right|_{\bm{\theta}=\bm{\theta}_{j,n}^*}
=
\sum_{\mathbf{x}\in P_{j,n}^*}
r_n(\mathbf{x})\,
\nabla_{\bm{\theta}}
\ell_n(\mathbf{x}|\bm{\theta}_{j,n}^*).
\end{align}
For each \(\mathbf{x}\in P_{j,n}^*\), the~score linearisation gives
\vspace{6pt}
\[
\nabla_{\bm{\theta}}\ell_n(\mathbf{x}\mid\bm{\theta}_{j,n}^*)
=
-n{\bf J}_1(\bm{\theta}_{j,n}^*)
\{\bm{\theta}_{j,n}^*-\hat{\bm{\theta}}(\mathbf{x})\}
+
\mathbf{r}_{j,n}(\mathbf{x}),
\]
where
\[
\sup_{\mathbf{x}\in P_{j,n}^*}
\|\mathbf{r}_{j,n}(\mathbf{x})\|
=
o(\sqrt n).
\]
Substituting this into the first-order condition
\[
0=\sum_{{\bf x}\in P^*_{j,n}} r_n({\bf x})\nabla_{\bm{\theta}} \ell_n({\bf x} \mid \bm{\theta}^*_{j,n})
\]
and dividing by \(nq^*_{j,n}\) yields
\[
{\bf J}_1(\bm{\theta}^*_{j,n})
\left(
\bm{\theta}^*_{j,n}
-
\sum_{{\bf x}\in P^*_{j,n}} w_{j,n}({\bf x})\hat{\bm{\theta}}({\bf x})
\right)
=
o(n^{-1/2}),
\]
where \(w_{j,n}({\bf x})=r_n({\bf x})/q^*_{j,n}\). Since \({\bf J}_1(\bm{\theta}^*_{j,n})\) is uniformly non-singular on compact subsets, we obtain
\[
\bm{\theta}^*_{j,n}
=
\sum_{{\bf x}\in P^*_{j,n}} w_{j,n}({\bf x})\hat{\bm{\theta}}({\bf x})
+
o(n^{-1/2}),
\]
as claimed.
\end{proof}
\section{Consistency}
\label{sec:consistency}
The preceding section shows that, in~the high-resolution regime of \ref{ass:A7}, SMML behaves locally as an entropy-weighted Fisher--Rao quantisation of the statistical model. 
%
The selected SMML codepoint remains within the same local \(n^{-1/2}\)-scale neighbourhood as the MLE.
Consequently, replacing the continuous MLE by the finite SMML codepoint perturbs estimation only at the usual parametric
scale.

\begin{theorem}[Consistency and rate of convergence of the SMML estimator]
\label{thm:smml:consistency}
Let
\[
\mathcal{P}_n^*
=
\{P_{j,n}^*\}_{j=1}^{k_n},
\qquad
\bm{\Theta}_n^*
=
\{\bm{\theta}_{j,n}^*\}_{j=1}^{k_n}
\]
denote an optimal SMML partition and its associated codepoints. Define the SMML estimator by
\[
\hat{\bm{\theta}}_n^{\mathrm{SMML}}(\mathbf{x})
=
\sum_{j=1}^{k_n}
\bm{\theta}_{j,n}^*
\mathds{1}_{\{\mathbf{x}\in P_{j,n}^*\}}.
\]
Let
\[
\hat{\bm{\theta}}_n
=
\hat{\bm{\theta}}(\mathbf{X}_n)
\]
\textls[-15]{denote the MLE, and~suppose that the true parameter is \(\bm{\theta}_0\in\operatorname{int}(\Theta)\).
%
Under~\ref{ass:A1}--\ref{ass:A7}, fix compact sets \(K_0\Subset K_1\Subset\Theta\), such that \(\theta_0\in \operatorname{int}(K_0)\) and (A5) holds for \(K_1\). Assume, in~addition, that the selected SMML codepoint lies in \(K_1\) with a probability tending to one whenever the MLE lies in \(K_0\), that is,}
\[
P\left(
\hat\theta_n\in K_0,\,
\hat\theta^{\mathrm{SMML}}_n\notin K_1
\right)\to 0.
\]
Then
\[
\hat{\bm{\theta}}_n^{\mathrm{SMML}}
\xrightarrow{p}
\bm{\theta}_0,
\qquad n\to\infty.
\]
Moreover,
\begin{align}
\label{eqn:smml:rate}
\left\|
\hat{\bm{\theta}}_n^{\mathrm{SMML}}
-
\bm{\theta}_0
\right\|
=
O_p(n^{-1/2}).
\end{align}
\end{theorem}

\begin{proof}
Write
\[
\hat{\bm{\theta}}_n^{\mathrm{SMML}}
:=
\hat{\bm{\theta}}_n^{\mathrm{SMML}}(\mathbf{X}_n),
\qquad
\hat{\bm{\theta}}_n
:=
\hat{\bm{\theta}}(\mathbf{X}_n).
\]
By standard likelihood theory under \ref{ass:A1}--\ref{ass:A6},
\[
\hat{\bm{\theta}}_n\stackrel{p}{\to}\bm{\theta}_0,
\qquad
\|\hat{\bm{\theta}}_n-\bm{\theta}_0\|=O_p(n^{-1/2}).
\]
Since \(\bm{\theta}_0\in \operatorname{int}(K_0)\), it follows that
\[
P(\hat{\bm{\theta}}_n\in K_0)\to 1.
\]
By the additional compact-localisation assumption,
\[
P\left(
\hat{\bm{\theta}}_n\in K_0,\,
\hat{\bm{\theta}}^{\mathrm{SMML}}_n\in K_1
\right)\to 1.
\]
%
On this event, \ref{ass:A7} gives
\[
d_F\!\left(\hat{\bm{\theta}}_n,\hat{\bm{\theta}}_n^{\mathrm{SMML}}\right)
=
O_p(n^{-1/2}).
\]
By \ref{ass:A5}, the~Fisher--Rao and Euclidean metrics are locally equivalent   on \(K_1\). Therefore,
\[
\|\hat{\bm{\theta}}^{\mathrm{SMML}}_n-\hat{\bm{\theta}}_n\|=O_p(n^{-1/2}).
\]
Hence,
\[
\|\hat{\bm{\theta}}_n^{\mathrm{SMML}}-\bm{\theta}_0\|
\le
\|\hat{\bm{\theta}}_n^{\mathrm{SMML}}-\hat{\bm{\theta}}_n\|
+
\|\hat{\bm{\theta}}_n-\bm{\theta}_0\|
=
O_p(n^{-1/2}),
\]
which proves \eqref{eqn:smml:rate}. Since \(O_p(n^{-1/2})=o_p(1)\),
consistency follows.
\end{proof}

\begin{remark}[Interpretation]
{Theorem~\ref{thm:smml:consistency} shows that, under~the high-resolution
regime of \mbox{\ref{ass:A7}}, replacing the continuous MLE by the
selected SMML codepoint introduces only an \(O_p(n^{-1/2})\) perturbation.}
\end{remark}
\section{SMML and Exponential~Families}
\label{sec:expfam}
We now specialise the general SMML framework to regular exponential family models. In~this setting, the~entropy--cross-entropy structure of the SMML objective has an explicit information-geometric form. Exponential families are dually flat statistical manifolds. The~Fisher information defines the Riemannian metric, natural and expectation parameters provide dual affine coordinates, and~KL divergence is the Bregman divergence generated by the log-partition function~\cite{Amari16}. Consequently, SMML codepoints and cells can be described in terms of KL projections, Bregman centroids, and~affine Voronoi-type~partitions.

The closest SMML-specific prior work is Dowty~\cite{Dowty13},
who studied exponential families with continuous sufficient statistics and showed that SMML regions are convex polytopes determined by the assertions and coding probabilities. The~results below give the corresponding countable-data formulation and connect it to the asymptotic Fisher--Rao and KL-projection framework developed in the preceding~sections.

We consider models whose joint likelihood for a data set \(\mathbf{x}\) of size \(n\) has the form
\begin{align}
\label{eqn:expfam:model}
p_n(\mathbf{x}\mid \bm{\theta})
=
h(\mathbf{x})
\exp\!\left\{
\eta(\bm{\theta})^{\prime}T(\mathbf{x})
-
A_n\bigl(\eta(\bm{\theta})\bigr)
\right\},
\end{align}
where \(T(\mathbf{x})\in\mathbb{R}^d\) is a sufficient statistic, \(\eta(\bm{\theta})\in\mathbb{R}^d\) is the natural parameter, and~\(A_n\) is the log-partition function for the full sample. The~model is canonical when \(\eta(\bm{\theta})=\bm{\theta}\).

\subsection{SMML Codepoints as KL/Bregman~Centroids}
For a fixed SMML cell \(P_j\), the~associated codepoint satisfies
\begin{align}
\label{eqn:expfam:smml:def}
\bm{\theta}_j^*
=
\argmax_{\bm{\theta}\in\Theta}
\sum_{\mathbf{x}\in P_j}
r_n(\mathbf{x})\log p_n(\mathbf{x}\mid\bm{\theta}).
\end{align}
Substituting~\eqref{eqn:expfam:model} and dropping the term involving \(\log h(\mathbf{x})\), which does not depend on \(\bm{\theta}\), gives
\begin{align}
\label{eqn:expfam:objective}
\bm{\theta}_j^*
=
\argmax_{\bm{\theta}\in\Theta}
\left\{
\eta(\bm{\theta})^{\prime}S_j
-
q_j A_n\bigl(\eta(\bm{\theta})\bigr)
\right\},
\end{align}
where
\[
S_j
=
\sum_{\mathbf{x}\in P_j}
r_n(\mathbf{x})T(\mathbf{x}),
\qquad
q_j
=
\sum_{\mathbf{x}\in P_j}r_n(\mathbf{x}) > 0.
\]
The first-order condition is
\begin{align}
\label{eqn:expfam:foc:general}
[D\eta(\bm{\theta}_j^*)]^{\prime}
\left\{
S_j
-
q_j\nabla_{\eta}A_n\bigl(\eta(\bm{\theta}_j^*)\bigr)
\right\}
=
0.
\end{align}
In the canonical case, this reduces to
\begin{align}
\label{eqn:expfam:canonical:foc}
\nabla A_n(\bm{\theta}_j^*)
=
\frac{S_j}{q_j}.
\end{align}
Since \(\nabla A_n(\bm{\theta})=\mathbb{E}_{\bm{\theta}}[T(\mathbf{X})]\), this is a moment-matching condition. The~expectation parameter of the SMML codepoint equals the \(r_n\)-weighted average of the sufficient statistic over the~cell.

\begin{proposition}[SMML codepoints in canonical exponential families]
\label{prop:smml:expfam}
Suppose that~\eqref{eqn:expfam:model} is a regular canonical exponential family, and~suppose that the cellwise mean \(S_j/q_j\) lies in the interior of the convex support of the sufficient statistic \(T\), so that the maximiser in \eqref{eqn:expfam:smml:def} exists, is unique, and~lies in \(\Theta\). Then, for~each SMML cell \(P_j\),
%
\begin{align}
\label{eqn:expfam:meanvalue}
\mathbb{E}_{\bm{\theta}_j^*}[T(\mathbf{X})]
=
\sum_{\mathbf{x}\in P_j}
w_j(\mathbf{x})T(\mathbf{x}),
\qquad
w_j(\mathbf{x})
=
\frac{r_n(\mathbf{x})}{q_j}.
\end{align}
Thus, the expectation parameter associated with the SMML codepoint is the cellwise \(r_n\)-weighted average of the sufficient statistic.
\end{proposition}
\begin{proof}
In the canonical case, \(\eta(\bm{\theta})=\bm{\theta}\), so the first-order condition in~\eqref{eqn:expfam:foc:general} reduces to
\[
\nabla A_n(\bm{\theta}_j^*)=\frac{S_j}{q_j}.
\]
Since
\[
S_j=\sum_{\mathbf{x}\in P_j} r_n(\mathbf{x})T(\mathbf{x}),
\qquad
q_j=\sum_{\mathbf{x}\in P_j} r_n(\mathbf{x}),
\]
we obtain
\[
\frac{S_j}{q_j}
=
\sum_{\mathbf{x}\in P_j}
\frac{r_n(\mathbf{x})}{q_j}\,T(\mathbf{x})
=
\sum_{\mathbf{x}\in P_j} w_j(\mathbf{x})T(\mathbf{x}).
\]
Finally, for~a regular canonical exponential family,
\[
\nabla A_n(\bm{\theta})=\mathbb{E}_{\bm{\theta}}[T(\mathbf{X})],
\]
which yields~\eqref{eqn:expfam:meanvalue}.
\end{proof}

The moment-matching condition in~\eqref{eqn:expfam:canonical:foc} is the standard characterisation of the KL (or I-) projection onto a regular exponential family~\cite{Csiszar75}. It is also the countable-data analogue of Dowty's centroid condition for SMML estimators with continuous sufficient statistics~\cite{Dowty13}. Since KL divergence in a canonical exponential family is the Bregman divergence generated by \(A_n\), Proposition~\ref{prop:smml:expfam} identifies SMML codepoints as KL/Bregman centroids in expectation coordinates, connecting the result to Bregman clustering theory~\cite{BanerjeeEtAl05}. Unlike standard Bregman clustering, however, the~cells and weights here arise from the SMML coding~criterion.
In general, the~natural parameter itself need not be an affine average of the corresponding MLEs. Such an interpretation holds only in special cases where the inverse mean map is affine and the MLE depends affinely on the sufficient~statistic.

\subsection{Polyhedral Structure of SMML~Cells}
The Fisher--Rao Voronoi geometry of Section~\ref{sec:fisherrao} is an asymptotic local result for regular parametric models. In~exponential families, an~exact finite-sample geometry is available because the log-likelihood is affine in the sufficient statistic. The~next result shows that exact SMML cells are pullbacks of convex polyhedra in sufficient-statistic
space. 

\begin{theorem}[Polyhedral SMML cells in sufficient-statistic space]
\label{thm:smml:polyhedra}
Let \(\mathcal{P}^*=\{P_1^*,\ldots,P_k^*\}\) be an optimal SMML partition for the exponential family~\eqref{eqn:expfam:model}, with~associated codepoints   \(\bm{\Theta}^*=\{\bm{\theta}_1^*,\ldots,\bm{\theta}_k^*\}\) and assertion probabilities \(q_1^*,\ldots,q^*_k\). For~each \(j\), define
\begin{align}
\label{eqn:expfam:region}
\mathcal{V}_j
=
\bigcap_{\ell\neq j}
\left\{
\mathbf{t}\in\mathbb{R}^d:
\bigl(\eta(\bm{\theta}_j^*)-\eta(\bm{\theta}_\ell^*)\bigr)^{\prime}\mathbf{t}
\ge
\log q_\ell^*-\log q_j^*
+
A_n\bigl(\eta(\bm{\theta}_j^*)\bigr)
-
A_n\bigl(\eta(\bm{\theta}_\ell^*)\bigr)
\right\}.
\end{align}
Then each \(\mathcal{V}_j\) is a convex polyhedron in \(\mathbb{R}^d\), and,
up to ties on boundaries,
\[
P_j^*
=
\{\mathbf{x}\in\mathcal{X}_n:T(\mathbf{x})\in\mathcal{V}_j\}.
\]
\end{theorem}

\begin{proof}
For a fixed codebook and assertion probabilities, assigning \(\mathbf{x}\) to assertion \(j\)  incurs the codelength
\[
\Lambda_j(\mathbf{x})
=
-\log q_j^*
-
\log p_n(\mathbf{x}\mid\bm{\theta}_j^*).
\]
The optimal partition assigns \(\mathbf{x}\) to an index minimising \(\Lambda_j(\mathbf{x})\), up~to ties. Thus, \(\mathbf{x}\in P_j^*\) precisely when
\[
\Lambda_j(\mathbf{x})\le \Lambda_\ell(\mathbf{x})
\qquad
\text{for all } \ell\neq j.
\]
Substituting the exponential-family representation and cancelling the common base-measure term \(\log h(\mathbf{x})\), this inequality becomes
\[
\bigl(\eta(\bm{\theta}_j^*)-\eta(\bm{\theta}_\ell^*)\bigr)^\prime
T(\mathbf{x})
\ge
\log q_\ell^* - \log q_j^*
+
A_n\bigl(\eta(\bm{\theta}_j^*)\bigr)
-
A_n\bigl(\eta(\bm{\theta}_\ell^*)\bigr).
\]
For fixed \(j\) and \(\ell\), this is a closed half-space in
sufficient-statistic space. Intersecting these half-spaces over all
\(\ell\neq j\) gives \(\mathcal{V}_j\). Hence, \(\mathcal{V}_j\) is a convex
polyhedron. Since the membership condition depends on \(\mathbf{x}\) only
through \(T(\mathbf{x})\), the~SMML cell is the pullback
\[
P_j^*
=
\{\mathbf{x}\in\mathcal{X}_n:T(\mathbf{x})\in\mathcal{V}_j\},
\]
up to boundary ties.
\end{proof}

Unlike Theorems \ref{thm:smml:fr:equivalence}--\ref{thm:smml:consistency}, Theorem~\ref{thm:smml:polyhedra} is an exact finite-sample result for exponential families. It follows directly from affine comparisons of two-part codelengths in sufficient-statistic space and does not rely on the high-resolution Assumption \ref{ass:A7}.

Figure~\ref{fig:smml:polyhedral} illustrates the exact exponential-family structure in Theorem~\ref{thm:smml:polyhedra} where the pairwise SMML codelength comparisons become affine boundaries in sufficient-statistic space, and~the resulting data-space cells are pullbacks under the sufficient-statistic map. Theorem~\ref{thm:smml:polyhedra} is the countable-data analogue of Dowty's polytope theorem for SMML estimators with continuous sufficient statistics~\cite{Dowty13}. In~both settings, pairwise comparisons of two-part codelengths become affine inequalities in sufficient-statistic space; the assertion probabilities provide the offsets, while differences between natural parameters determine the normal vectors of the separating hyperplanes. This affine structure also agrees with the general geometry of Bregman Voronoi diagrams, where nearest-codepoint comparisons for Bregman divergences become affine inequalities in suitable dual coordinates~\cite{BoissonnatEtAl10}.

\begin{figure}
\begin{center}
\begin{tikzpicture}[
    >=Latex,
    font=\small,
    axis/.style={->, line width=0.8pt, black!75},
    boundary/.style={draw=black!65, dashed, line width=0.9pt},
    panel/.style={draw=black!70, line width=0.7pt, rounded corners=2pt},
    statpoint/.style={circle, fill=black!45, inner sep=1.15pt},
    centroid/.style={circle, fill=black, draw=white, line width=0.4pt, inner sep=2.0pt},
    smalllab/.style={font=\scriptsize},
    lab/.style={font=\small}
]

\def\xmin{0}
\def\xmax{6.4}
\def\ymin{0}
\def\ymax{4.2}

\begin{scope}
\clip (\xmin,\ymin) rectangle (\xmax,\ymax);

\fill[blue!12]
    (\xmin,\ymin) -- (2.15,\ymin) -- (2.55,\ymax) -- (\xmin,\ymax) -- cycle;

\fill[green!12]
    (2.15,\ymin) -- (\xmax,\ymin) -- (\xmax,\ymax) -- (3.85,\ymax)
    -- (2.55,\ymax) -- cycle;

\fill[orange!18]
    (2.15,\ymin) -- (3.85,\ymax) -- (2.55,\ymax) -- cycle;

\draw[boundary] (2.15,\ymin) -- (2.55,\ymax);
\draw[boundary] (2.15,\ymin) -- (3.85,\ymax);

\end{scope}

\draw[panel] (\xmin,\ymin) rectangle (\xmax,\ymax);

\draw[axis] (-0.15,\ymin) -- (\xmax+0.25,\ymin)
    node[right] {$T_1(\mathbf{x})$};
\draw[axis] (\xmin,-0.15) -- (\xmin,\ymax+0.25)
    node[above] {$T_2(\mathbf{x})$};

\foreach \x/\y in {
0.55/0.55, 0.80/1.25, 1.05/2.05, 1.35/3.00,
1.55/0.75, 1.75/1.65,
2.35/0.85, 2.55/3.05, 2.65/1.75, 3.10/3.45, 3.35/3.20,
3.65/0.55, 3.95/1.10, 4.30/1.85, 4.70/2.65,
5.10/0.70, 5.35/1.55, 5.70/2.45, 5.95/3.25
}
{
\node[statpoint] at (\x,\y) {};
}

\node[centroid] (c1) at (1.05,2.55) {};
\node[centroid] (c2) at (4.85,2.15) {};
\node[centroid] (c3) at (2.75,2.75) {};

\node[smalllab, anchor=south east] at ($(c1)+(0.05,-0.05)$)
    {$\mu_1^*$};
\node[smalllab, anchor=south west] at ($(c2)+(0.05,-0.2)$)
    {$\mu_2^*$};
\node[smalllab, anchor=north west] at ($(c3)+(-0.20,-0.02)$)
    {$\mu_3^*$};

\node[smalllab, blue!55!black] at (0.85,3.75) {$\mathcal V_1$};
\node[smalllab, green!45!black] at (5.45,3.75) {$\mathcal V_2$};
\node[smalllab, orange!70!black] at (3.05,3.75) {$\mathcal V_3$};

\node[smalllab, align=center, fill=white, inner sep=1.5pt] at (4.25,0.45)
{affine SMML\\decision boundaries};

\draw[->, black!60, line width=0.6pt]
    (4.05,0.70) -- (2.95,1.75);

\node[lab, align=center] at (3.20,-0.85)
{
\(\displaystyle
P_j^*
=
\{\mathbf{x}\in\mathcal X_n:T(\mathbf{x})\in\mathcal V_j\}
=
T^{-1}(\mathcal V_j)
\)
};

\node[lab] at (3.2,4.55)
{Sufficient-statistic space};
\end{tikzpicture}
\caption{
Exact 
 SMML cells in sufficient-statistic space for an exponential family. Pairwise comparisons of two-part codelengths become affine inequalities in \(T(\mathbf{x})\), so each region \(\mathcal V_j\) is a convex polyhedron. The~grey points represent attainable sufficient-statistic values, while the black points indicate the corresponding expectation-coordinate centroids \(\mu_j^*\). The~SMML cell in data space is the pullback \(P_j^*=T^{-1}(\mathcal V_j)\).
}
\label{fig:smml:polyhedral}
\end{center}
\end{figure}

For multinomial models, the~structure is especially transparent. In~the binomial case, the~sufficient statistic is one-dimensional, so each SMML cell is an interval of count values. For~a \(K\)-category multinomial model, each pairwise SMML boundary is an affine hyperplane in the count vector, and~each cell is obtained by intersecting the multinomial lattice with a finite collection of half-spaces. For~a fixed finite codebook, this is a weighted analogue of the logarithmic Voronoi partitions of Alexandr and Heaton~\cite{AlexandrHeaton21}. 
%
%

%

\subsection{Example: Poisson~Model}
The Poisson model is a one-dimensional example of both the centroid condition in Proposition~\ref{prop:smml:expfam} and the exact affine cell structure in Theorem~\ref{thm:smml:polyhedra}. Let \(X_1,\dots,X_n\) be independent Poisson observations with common mean \(\lambda > 0\), and~write
\[
T(\mathbf{x})=\sum_{i=1}^n x_i
\]
for the sufficient statistic of the sample \({\bf x}=(x_1,\dots,x_n)\). The~joint likelihood is

\[
p_n( {\bf x}\mid \lambda)
=
\left(\prod_{i=1}^n \frac{1}{x_i!}\right)
\exp\{T({\bf x})\log\lambda - n\lambda\},
\]
so the model is a regular canonical exponential family with natural parameter
\[
\eta=\log\lambda,
\]
sufficient statistic \(T({\bf x})\), and~log-partition function
\[
A_n(\eta)=ne^\eta=n\lambda.
\]
For a fixed SMML cell \(P_j\), Proposition~\ref{prop:smml:expfam} gives the moment-matching condition
\[
E_{\lambda_j^*}[T({\bf X})]
=
\sum_{{\bf x}\in P_j} w_j({\bf x}) T({\bf x}),
\qquad
w_j({\bf x})=\frac{r_n({\bf x})}{q_j}.
\]
Since \(E_\lambda[T({\bf X})]=n\lambda\), this becomes
\[
n\lambda_j^*
=
E_{r_n}[T({\bf X})\mid {\bf X}\in P_j],
\]
and hence
\[
\lambda_j^*
=
\frac{1}{n}E_{r_n}[T({\bf X})\mid {\bf X}\in P_j].
\]
Thus, the Poisson SMML codepoint is the marginal-predictive conditional expectation of the sample mean over the cell, which is the Poisson instance of the KL/Bregman centroid condition in Proposition~\ref{prop:smml:expfam}.

The exact SMML partition is equally explicit. For~a fixed codebook \(\{\lambda_1,\dots,\lambda_k\}\) and assertion probabilities \(\{q_1,\dots,q_k\}\), the~two-part codelength for assigning sample \({\bf x}\) to assertion \(j\) is
\[
\Lambda_j({\bf x})
=
-\log q_j - \log p_n({\bf x}\mid \lambda_j).
\]
Comparing assertions \(j\) and \(\ell\), and~cancelling the common base measure term, gives
\[
(\log\lambda_j-\log\lambda_\ell) T({\bf x})
\ge
\log q_\ell-\log q_j + n(\lambda_j-\lambda_\ell).
\]
Hence, each exact SMML cell is the pullback of an interval in the sufficient statistic \(T({\bf x})\), and~the pairwise threshold is
\[
t_{j\ell}
=
\frac{\log q_\ell-\log q_j + n(\lambda_j-\lambda_\ell)}
{\log\lambda_j-\log\lambda_\ell},
\qquad
\lambda_j\neq \lambda_\ell ,
\]
where the inequality defining membership of cell \(j\) is
\(T(\mathbf{x})\ge t_{j\ell}\) when \(\lambda_j>\lambda_\ell\) and
\(T(\mathbf{x})\le t_{j\ell}\) when \(\lambda_j<\lambda_\ell\).
In particular, after~ordering the assertions by their Poisson means, the~exact SMML partition is obtained by intersecting the integer lattice of attainable sufficient-statistic values with a finite collection of threshold intervals. This is the one-dimensional Poisson instance in Theorem~\ref{thm:smml:polyhedra}.

The asymptotic Fisher--Rao geometry of Section~\ref{sec:fisherrao} is also especially simple in this example. For~one Poisson observation,
\[
J_1(\lambda)=\frac{1}{\lambda},
\]
so the per-observation Fisher--Rao metric is

\[
ds^2=\frac{d\lambda^2}{\lambda}.
\]
The corresponding Fisher--Rao distance is~\cite{MiyamotoEtAl24}
\vspace{8pt}
\[
d_F(\lambda_1,\lambda_2)=2|\sqrt{\lambda_1}-\sqrt{\lambda_2}|.
\]
Thus, in~the high-resolution regime of Theorem~\ref{thm:smml:fr:equivalence}, local SMML assignment compares nearby Poisson assertions using
\[
4(\sqrt{\lambda}-\sqrt{\lambda_j})^2-\frac{2}{n}\log q_j,
\]
showing explicitly how assertion probabilities shift the Voronoi boundaries. The~finite-sample threshold structure in the sufficient statistic \(T({\bf x})\) and the asymptotic weighted Voronoi structure in \(\lambda\)-space are therefore consistent descriptions of the same underlying SMML geometry at different levels of~approximation.
\section{Discussion}
\label{sec:discussion}
This paper has developed an information-theoretic and geometric perspective on strict minimum message length (SMML) estimation in regular parametric models. The~central insight is that SMML is an entropy--cross-entropy optimisation principle that produces a finite-resolution approximation of a continuous statistical model. In~the high-resolution regime, this approximation acquires a local geometry governed by Kullback--Leibler (KL) divergence and the Fisher~information.

At the level of partitions, the~SMML decision rule is asymptotically equivalent to a weighted Fisher--Rao Voronoi rule in parameter space, pulled back through the maximum likelihood estimator (MLE). The~assertion probabilities appear as additive weights, reflecting the contribution of assertion entropy to the two-part codelength. At~the level of codepoints, each SMML assertion is the KL projection of the normalised cellwise predictive distribution onto the model family. Thus, SMML simultaneously quantises the parameter space and performs local KL-optimal fitting within each quantisation~cell.

The finite assertion code is first-order asymptotically lossless in the sense that, under~the high-resolution regime of~\ref{ass:A7}, replacing the MLE by the selected SMML codepoint introduces only an \(O_p(n^{-1/2})\) perturbation. Hence, the SMML coding constraint does not alter classical likelihood asymptotics at first order. Rather, it imposes a structured entropy-coded discretisation whose geometry is
determined by the underlying statistical model.
For regular exponential families, the~geometry becomes exact. In~canonical form, SMML codepoints satisfy a moment-matching condition whereby the expectation parameter of the assertion equals the marginal-predictive weighted average of the sufficient statistic in the cell. Equivalently, SMML codepoints are KL/Bregman centroids, connecting the codepoint condition to Bregman centroid and clustering theory~\cite{BanerjeeEtAl05}. Pairwise codelength comparisons also yield affine inequalities in sufficient-statistic space, so that the exact SMML partition is the pullback of a polyhedral partition. This is the countable-data analogue of Dowty's polytope theorem for SMML estimators with continuous sufficient statistics~\cite{Dowty13}, and~is consistent with the affine structure of Bregman Voronoi diagrams~\cite{BoissonnatEtAl10}. In~the SMML setting, the~assertion probabilities provide the affine offsets in these Voronoi-type~boundaries.

The computational difficulty of exact SMML optimisation also suggests a role for geometric approximations. Farr and Wallace~\cite{FarrWallace02} proved that SMML inference is NP-hard via a reduction from Exact Cover by 3-Sets (X3C), although~special one-dimensional cases admit efficient algorithms. The~Fisher--Rao and Bregman Voronoi structures identified here may therefore provide useful guidance for practical SMML codebook-construction procedures. Theorem~\ref{thm:smml:fr:equivalence} implies that, in~the high-resolution regime, the~assignment step is approximately a weighted Fisher--Rao Voronoi rule. 
Theorem~\ref{thm:smml:weighted:mle} further suggests a corresponding update step, in~which each codepoint is replaced by the weighted average of the MLEs assigned to its cell, or, more generally, by~the fixed-cell KL projection of the conditional marginal distribution. This leads naturally to alternating schemes for approximate SMML optimisation. In~exponential families, Proposition~\ref{prop:smml:expfam} and Theorem~\ref{thm:smml:polyhedra} show that assignment reduces to threshold or affine half-space comparisons in sufficient-statistic space, while codepoint updates reduce to moment matching. Such procedures do not remove the global NP-hardness of exact SMML optimisation, but~they provide a practical way to exploit the induced Fisher--Rao or Bregman geometry when constructing approximate SMML~codebooks.

A limitation of the present analysis is that the high-resolution LAN-scale regime is assumed rather than derived from the global SMML optimisation problem. The~results of Sections~\ref{sec:fisherrao} and \ref{sec:consistency} therefore describe the local geometry of SMML once the finite assertion codebook resolves the model at the LAN \(n^{-1/2}\) scale. A~local split-optimality heuristic supports this scale, since the entropy cost of splitting a cell of mass \(q_j\) is of order \(q_j\), whereas the local quadratic reduction in detail codelength is of order \(q_j n r_j^2\) for a cell of effective Fisher--Rao radius \(r_j\). Establishing a rigorous global derivation of this regime for exact SMML optima remains an important direction for future~work.

Overall, the~results show that SMML provides an information-geometric bridge between statistical inference, quantisation, and~compression. Starting from an entropy-based coding criterion, it induces KL projections, Bregman centroids, and~Voronoi-type geometry in both asymptotic and exponential-family settings. In~practical terms, the~results suggest geometry-aware approximate SMML procedures for finite codebook design, model discretisation, and~finite-resolution statistical compression in regular parametric families.

\bibliographystyle{unsrt}  
\bibliography{bibliography}  

\end{document}